\documentclass{amsart}
\usepackage{graphicx} % Required for inserting images
\usepackage{amsmath}
\usepackage{amsfonts}
\usepackage{amsthm}
\usepackage{enumitem}
\usepackage{mathrsfs}
\usepackage[T1]{fontenc}
\usepackage{amssymb}
\usepackage{tikz}

\usepackage[looser]{newpxtext}
\usepackage[smallerops]{newpxmath}
\usepackage{bm}% load after all math to give access to bold math
\newcommand{\XT}{X_T}

\newcommand{\XTd}{X^{\delta}_T}
\newcommand{\XZ}{X^\Z}

\newcommand{\barrH}{\bar{\rho}^H}

\newcommand{\meas}{\mathcal{M}}

\newcommand{\MT}{\meas_T}
\newcommand{\MTX}{\meas_T(X)}

\newcommand{\MSeXZ}{\meas^e_S(X^\Z)}
\newcommand{\MS}{\meas_S}
\newcommand{\MSe}{\meas_S^e}

\newcommand{\trho}{\tilde{\rho}}
\newcommand{\rB}{\rho_B}
\newcommand{\crB}{\trho}
\newcommand{\rBH}{\rho_B^H}
\newcommand{\rZ}{\pi}%{\rho^\Z}
\newcommand{\brZ}{\bar{\rZ}}
\newcommand{\brZH}{\brZ^H}

\newcommand{\rZB}{\rZ_B}
\newcommand{\hatrB}{\hat\rho_B}
\newcommand{\hatrBZ}{\hat{\rZ}_B}
\newcommand{\rZBH}{\rZ_B^H}
\newcommand{\dbar}{\bar d}
\newcommand{\cs}{\mathrm{C}^*}
\newcommand{\Ell}{\mathrm{Ell}}
\newcommand{\tr}{\mathrm{Tr}}

\newcommand{\N}{\mathbb{N}}
\newcommand{\Z}{\mathbb{Z}}

\newcommand{\R}{\mathbb{R}}
\newcommand{\eps}{\varepsilon}
\DeclareMathOperator{\diam}{diam}

\DeclareMathOperator{\proj}{proj}
\newcommand{\projX}{\proj^1}
\newcommand{\projY}{\proj^2}
\newcommand{\projZ}{\proj^i}

\newtheorem{thm}{Theorem}[section] % 1st argument is your name for it
\newtheorem{lemma}[thm]{Lemma}     % 2nd argument is what is printed
\newtheorem{cor}[thm]{Corollary}
\newtheorem{prop}[thm]{Proposition}

\theoremstyle{definition}
\newtheorem{defn}[thm]{Definition}

\title{Even the vague specification property implies density of ergodic measures}

\author{Damla Bulda\u{g}}
\address{Faculty of Mathematics and Computer Science, Jagiellonian University in Krakow, ul. {\L}ojasiewicza 6, 30-348 Krak\'{o}w, Poland}
\email{damla.buldag.db@gmail.com}

\author{Bhishan Jacelon}
\address{
Institute of Mathematics of the Czech Academy of Sciences\\ \v{Z}itn\'{a} 25 \\115 67 Prague 1\\Czech Republic}
\email{jacelon@math.cas.cz}

\author{Dominik Kwietniak}
\address{Faculty of Mathematics and Computer Science, Jagiellonian University in Krakow, ul. {\L}ojasiewicza 6, 30-348 Krak\'{o}w, Poland}
\email{dominik.kwietniak@uj.edu.pl}

\date{\today}

\begin{document}

\begin{abstract}
We prove that if a topological dynamical system $(X,T)$ is surjective and has the vague specification property, then its ergodic measures are dense in the space of all invariant measures. The vague specification property generalises Bowen's classical specification property and encompasses the majority of the extensions of the specification property introduced so far.

The proof proceeds by first considering the natural extension $\XT$ of $(X,T)$ as a subsystem of the shift action on the space $X^\Z$ of $X$-valued biinfinite sequences. We then construct a sequence of subsystems of $X^\Z$ that approximate $\XT$ in the Hausdorff metric induced by a metric compatible with the product topology on $X^\Z$. The approximating subsystems consist of $\delta$-chains for $\delta$ decreasing to $0$. We show that chain mixing implies that each approximating system possesses the classical periodic specification property. Furthermore, we use vague specification to prove that our approximating subsystems of $X^\Z$ converge to $\XT$ in the Hausdorff metric induced the Besicovitch pseudometric. It follows that the simplices of invariant measures of these subsystems of $\delta$-chains converge to the simplex of invariant measures of $\XT$ with respect to a generalised version of Ornstein's $\bar{d}$ metric. What is more, the density of ergodic measures is preserved in the limit. The proof concludes by observing that the simplices of invariant measures for $\XT$ and $(X,T)$ coincide.

The approximation technique developed in this paper appears to be of independent interest. % beyond the main result.
\end{abstract}

\maketitle

\section{Introduction}

%Our understanding of properties implying that a topological dynamical system has dense ergodic measures evolved allowing us to prove that the phenomenon holds for increasingly general families of systems.

The problem of detecting density of ergodic measures among invariant measures of a topological dynamical system has a rich history dating back at least to 1939. That year, Jean Ville published in his Ph.D. thesis \cite{Ville} a result that in modern terminology says that ergodic measures are dense among invariant measures for the binary full shift. In 1961, Poulsen \cite{Poulsen} constructed the first explicit example of an infinite-dimensional Choquet simplex (an infinite-dimensional generalisation of the well-known $n$-dimensional simplex, which is the closed convex hull of the origin and canonical basis of $\R^n$), whose extreme points form a dense subset. An infinite convex set with this property seems to be quite remarkable and counterintuitive. 
In a landmark 1978 paper, Lindenstrauss, Olsen, and Sternfeld \cite{LOS78} proved that such a Choquet simplex is essentially unique: Any non-trivial (not one-point) simplex with dense extreme points must be affinely homeomorphic to the simplex discovered by Poulsen. Their work also revealed another remarkable property of the Poulsen simplex: every Choquet simplex appears as a face (closed convex hull of some subset of extreme points) of the Poulsen simplex, highlighting its role as a universal object in convex geometry. Hence, the result of Ville says that the simplex of invariant measures of the full binary shift is also an example of the Poulsen simplex.

The density of ergodic measures among all invariant measures has several important implications for topological dynamical systems. This property is crucial in various proofs, for example, when we study entropy and pressure of dynamical systems (see \cite{GK,KKK}). The first proofs of the density of ergodic measures, such as those given by Parthasarathy \cite{Par61} and Sigmund \cite{Sig70,Sig74} relied heavily on the specification property as the primary tool. 

The class of topological dynamical systems with the specification property includes: topologically mixing shifts of finite type, topologically mixing interval maps, and Axiom A diffeomorphisms restricted to mixing basic sets. However, there are many systems  (including diffeomorphisms and symbolic systems, such as $\beta$-shifts and $S$-gap shifts), without the specification property (see \cite{Dat90,KLO16,PS05}). 

%The connection back to dynamical systems was strengthened by Downarowicz's work in 1991, who proved that every Choquet simplex is affinely homeomorphic to the space of invariant measures for some minimal dynamical system. 

These limitations sparked the development of more sophisticated approaches. Mathematicians began exploring generalisations such as the weak specification property \cite{Marcus,Dat81,Dat83} and the almost specification property introduced by Pfister and Sullivan \cite{PS05}. 
In \cite{GK}, the authors (inspired by Coudene and Shapira \cite{CS}) introduced the properties of closability and linkability, providing a unified framework that encompasses an even broader class of topological dynamical systems while still yielding the crucial density result.
These properties properly generalise the periodic specification property. The weak specification property is satisfied by many systems such as all endomorphisms of compact Abelian groups for which the Haar measure is ergodic \cite{Dat90}. The almost specification property holds for any $\beta$-shift (see \cite{PS05}). It is known that neither the weak specification property implies the almost specification nor the almost specification property implies the weak specification property, see \cite{KLO17,KOR, Pavlov}. However, each of these properties implies the asymptotic average shadowing property; see \cite{KLO17}. The asymptotic average shadowing property was introduced by Gu \cite{Gu} and it turns out (see \cite{CKKK,CT} and below) that the asymptotic average shadowing property is still more general than the variants of the specification property mentioned so far. See Figure~\ref{fig:venn} for an illustration of dependencies. Each region of the diagram corresponds to a nonempty family of topological dynamical systems. As, up to this point, the evolution of approaches to density of ergodic measures followed a path from specification to more general properties of specification-type it is natural to ask if the ergodic measures are dense for a topological dynamical system with the asymptotic average shadowing property. The conjecture sounds plausible, especially in light of the fact that another result of Sigmund \cite{Sig74} saying that every invariant measure has a generic point in a topological dynamical system with the specification property has been successfully extended for systems with all the specification-like properties mentioned in Figure~\ref{fig:venn} (see \cite{Dat81,Dat83,PS05}), including the asymptotic average shadowing property (independently, in \cite{DTY, Kamae, KLO17}), see also \cite{CS,GK}. What is more, by some new developments the asymptotic average shadowing property turned out to be equivalent to the vague specification property introduced 45 years earlier by Kamae. Kamae  introduced the property in \cite{Kamae} and used it to prove results about lifting generic points. He used it to  study normal sequences and normality-preserving subsequences. Similar results were proved recently by Downarowicz and Weiss \cite{DW} under the assumption of the weak specification property. The question about the relation between the vague specification property and the weak specification property raised in \cite{DW} was resolved by Can and Trilles \cite{CT} who showed that the weak specification property implies the vague specification propety. But first, Can and Trilles showed that the vague specification property is equivalent to the asymptotic average shadowing property introduced independently by Gu \cite{Gu}. Using this equivalence Can and Trilles argued in \cite{CT} that the implication between the weak specification property and the vague specification property and non-equivalence of these properties follow from known results about the asymptotic average shadowing property \cite{CT,KLO17, Pavlov}. 

Unfortunately, for systems with the asymptotic average shadowing property (the vague specification property), it is impossible to follow the footsteps of Sigmund to prove density of ergodic measures. This is because the systems with the average shadowing property can be minimal (so they have no notrivial subsystems), while the weak specification property and the almost specification property implies the existence of many pairwise disjoint invariant subsets, which are used in proofs as supports of ergodic invariant measures approximating given non-ergodic measure (see \cite{Dat81,Dat83,KOR,PS05}). 

This obstacle has also been recently removed, but only for symbolic systems (shift spaces). In \cite{KKK}, the authors proposed an approximation scheme for shift spaces that provides a new method for proving entropy density of ergodic measures (a property stronger than just density of ergodic measures among all invariant ones). The method extends the list of shift spaces with entropy-dense (hence just dense) ergodic measures beyond previously known cases. It works by approximating  symbolic systems (shifts) with simpler ones while carefully controlling a certain pseudometric between the final system and its approximations.
This works as follows: for every shift space $X$, there is a canonical sequence of Markov approximations $X_n^M$ where each $X_n^M$ is a shift of finite type. The languages of $X_n^M$ are defined so that they match the language of $X$ for words up to length $n+1$. As $n$ increases, these approximations $X_n^M$ converge to $X$ in the sense of the Hausdorff metric induced on the compact subsets of the full shift by the usual (product) metric. The authors of \cite{KKK} observed that there is an abstract property of the shift space $X$, coined \emph{$\dbar$-shadowing} that guarantees that Markov approximations $X_n^M$ converge to $X$ in a stronger sense, the Hausdorff pseudometric $\dbar^H$ for subsets of the full shift induced by the $\dbar$ pseudometric on the full shift.
Then the authors of \cite{KKK} went on to prove that if the sequence $X^M_n$ converges to $X$ in the $\dbar$ pseudometric, then the corresponding spaces of invariant measures also converge.

Furthermore, the approximating spaces' ergodic measures converge to the limit space's ergodic measures. The convergence of simplices of invariant measures is in the sense of the metric $\dbar^H$, which is  Ornstein's metric $\dbar$ extended to subsets of invariant measures using the Hausdorff construction. The entropy function is continuous under this convergence. This allows one to transfer entropy density: if a sequence of shift spaces with entropy-dense ergodic measures converges to $X$ in the $\dbar^H$ sense, then $X$ also has entropy-dense ergodic measures. 
The power of this approach is that it allows proving entropy density for new classes of systems, including hereditary $\mathscr{B}$-free shifts and some minimal or proximal systems \cite{CKKK} that were previously out of reach of existing techniques.

In this paper, we would like to mimic the approach developed for shift spaces in \cite{KKK}. It turns out that many tools and pieces are already available in the literature, but we still need to construct analogous approximations for general topological dynamical systems and put together some seemingly unrelated pieces of theory. Note that for a topological dynamical system $(X,T)$ where $X$ is a compact metric space and $T$ is a continuous surjection, potential analogues of Markov approximations should have three key properties:
Each approximation should be a topological dynamical system $(X_n,T_n)$ with easily verifiable properties. We concentrate on the density of ergodic measures.  
The approximations should form a sequence that converges to the original system in some suitable (pseudo) metric. The pseudometric should pass the density of ergodic measures from approximations to the limit.

In order to find such approximations, we first pass to the natural extension $\XT$ of $(X,T)$. We regard $\XT$ as a subsystem of the shift action on the space $X^\Z$ of $X$-valued biinfinite sequences. We define our approximations to be subsystems of $X^\Z$ that converge to $\XT$ in the Hausdorff metric induced by the metric compatible with the product topology on $X^\Z$. For $\delta>0$ we choose the approximating subsystems $\XTd$ to be built from biinfinite $\delta$-chains for $T$, and then we let $\delta$ decrease to $0$. It turns out that chain mixing of $(X,T)$ implies that each approximating system $\XTd$ possesses the classical periodic specification property. In particular, the ergodic measures are dense in the simplex of shift-invariant measures concentrated on $\XTd$. Furthermore, we use the vague specification property to prove that our approximating subsystems of $X^\Z$ converge to $\XT$ in the Hausdorff metric induced by the Besicovitch pseudometric. It follows that the simplices of invariant measures of these subsystems of $\delta$-chains converge to the simplex of invariant measures of $\XT$ with respect to a generalised version of Ornstein's $\bar{d}$ metric. What is more, the density of ergodic measures is preserved in the limit. The proof concludes by observing that the simplices of invariant measures for $\XT$ and $(X,T)$ coincide.

We hope that the approximation technique developed in this paper is of independent interest. 
We also note that there are consequences in the field of operator algebras. It is well known that a homeomorphism $T\colon X\to X$ provides a dual action of integers on the space $C(X)$ of continuous functions $X\to\mathbb{C}$, which is encoded by the crossed product $\cs$-algebra $A=C(X)\rtimes_{T}\mathbb{Z}$. If $T$ is minimal, then $A$ is simple, and if in addition $X$ has finite Lebesgue covering dimension, then $A$ is included among the $\cs$-algebras classified by the Elliott invariant $\Ell$ (see \cite{Carrion:wz}). Computation of $\Ell(A)$ (and therefore the $\cs$-isomorphism class of $A$) entails, in part, knowledge of the Choquet simplex $\tr(A)$ of tracial states $A\to\mathbb{C}$. (In fact, by the classification presented in \cite{Elliott:2020wc} and with $\mathcal{W}$ the $\cs$-algebra constructed in \cite{Jacelon:2010fj}, $\tr(A)$ determines the `$\mathcal{W}$-stable' isomorphism class of $A$, that is, the isomorphism class of $A\otimes\mathcal{W}$.) If $T$ acts freely on $X$, then $\tr(A)$ coincides with $\MTX$ (see \cite[Chapter 11]{Giordano:2018wp}), but the systems of interest to us in the present article are of course not free. Whether minimality can be arranged is also not obvious. That said, by discarding the commutative $\cs$-algebra $C(X)$ for a suitable noncommutative (but Elliott-classifiable) one, and replacing invertibility, minimality and freeness of $T$ by `finite Rokhlin dimension' (see \cite{Hirshberg:2015wh,Szabo:2019te}), one can associate to $(X,T)$ an Elliott-classifiable `quantum' crossed product whose tracial state space is $\MTX$ (see \cite{Jacelon:2024aa}). A consequence of our main theorem is that there are exactly two $\mathcal{W}$-stable isomorphism classes for such quantum crossed products associated to systems with the vague specification property, one represented by $\mathcal{W}$ itself and the other by the tensor product of $\mathcal{W}$ with an approximately finite-dimensional $\cs$-algebra whose trace space is the Poulsen simplex.

\subsection*{Acknowledgements} DB was supported by the National Science Centre (NCN), Poland under Sonata grant no. 2022/47/D/ST1/02211. 
BJ was supported by the Czech Science Foundation (GA\v{C}R) project 22-07833K and the Institute of Mathematics of the Czech Academy of Sciences (RVO: 67985840). DK was supported by the National Science Centre (NCN), Poland
 under the Weave-UNISONO call in the Weave programme [grant no
 2021/03/Y/ST1/00072].

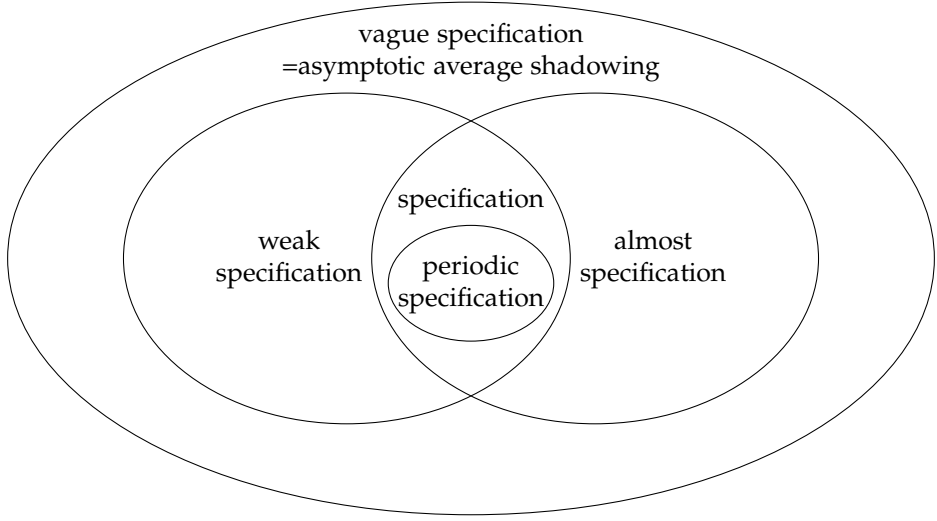
\begin{figure}
    \centering
\begin{tikzpicture}[scale=1.1]
    % Define styles for the shapes
    \tikzstyle{oval}=[ellipse, draw, minimum width=40pt, minimum height=20pt]
    
    % Draw the outer oval (asymptotic average shadowing = vague specification)
    \draw [draw] (0,0) ellipse [x radius=5.6cm, y radius=3.1cm];
    \node at (0,2.5) {\begin{tabular}{c}vague specification\\=asymptotic average shadowing  \end{tabular}};
    
    % Draw the two intersecting ovals (weak and almost specification)
    \draw [draw] (-1.5,0) ellipse [x radius=2.7cm, y radius=2cm];
    \draw [draw] (1.5,0) ellipse [x radius=2.7cm, y radius=2cm];
    
    % Labels for the two main intersecting ovals
    \node at (-2.2,0) {\begin{tabular}{c}weak\\specification\end{tabular}};
    \node at (2.2,0) {\begin{tabular}{c}almost\\specification\end{tabular}};
    
    % Label for the intersection
    \node at (0,0.7) {specification};
    
    % Draw and label the smallest oval (periodic specification)
    \draw [draw] (0,-0.3) ellipse [x radius=1cm, y radius=0.7cm];
    \node at (0,-0.3) {\begin{tabular}{c}periodic\\specification\end{tabular}};
\end{tikzpicture}    
    \caption{Diagram illustrating dependencies between various specification like properties discussed in the paper.}
    \label{fig:venn}
\end{figure}

\section{Definitions and notation}

%$X^\Z$
%Define the space $\XT$, $\XTn$ $\XTd$

Let $(X,\rho)$ be a compact metric space and $T\colon X\rightarrow X$ be a continuous surjective map. If necessary, we replace $\rho$ by an equivalent metric given by $\min(1,\rho(x,y))$. Hence, without loss of generality, for the rest of the paper we assume $\diam(X)\le 1$, that is, $\rho(x,y)\le 1$ for $x,y\in X$. We call $(X,T)$ a \emph{topological dynamical system} or a \emph{tds} for short. 

We write $\meas(X)$ for the space of all Borel probability measures on $X$. We say that a $\mu\in\meas(X)$ is $T$-invariant if for every Borel set $B\subseteq X$ we have $\mu(T^{-1}(B))=\mu(B)$. 
We say that a $\mu\in\meas(X)$ is \emph{ergodic} if for every Borel set $E\subseteq X$ such that $T^{-1}(E)=E$ we have $\mu(E)\in\{0,1\}$.

We denote the set of all bi-infinite $X$-valued sequences by $X^{\Z}$, that is, 
\[X^{\Z}= \{(x_{k})_{k\in\Z}: \forall k\in\Z, x_{k}\in X \}.\] 
We equip $X^\Z$ with the product topology, which is compatible with the metric
\[
\rZ((x_{k}),(y_{k}))= \sup\{\min(\rho(x_k, y_k), 1/(|k|+1)) : k \in \mathbb{Z}\}.
\]
%where we agree that $\min(d,1/0)=d$ for any $d\ge 0$ by convention.
%product metric $\rZ((x_{k}),(y_{k})) = \sum_{k\in\Z}2^{-|k|}\rho(x_k,y_k)$. 
Note that for every $\eps>0$ and $x=(x_{k})_{k\in\Z},\ y=(y_{k})_{k\in\Z}\in X^\Z$ we have that $\rZ((x_{k}),(y_{k}))<\eps$ if and only if  
\begin{equation}\label{cond:dist-bound}
\rho(x_k,y_k)<\eps\text{ for all $k\in \Z$ such that }|k|+1\le \max(1,{1}/{\eps}).    
\end{equation}

Let $S$ be the (left) shift operator on the product space $X^\Z$ given for $x=(x_{k})_{k\in\Z}$ by
\[
S(x)=(S(x)_k)_{k\in\Z},\quad\text{where }
S(x)_k=x_{k+1}\text{ for every }k\in\Z.
\]
We call every nonempty, closed and $S$-invariant set $Y\subseteq X^\Z$ a \emph{subsystem} of $X^\Z$ or a \emph{subshift}. 
Given a topological dynamical system $(X,T)$ we define the space of all (bi-infinite) orbits  of $(X,T)$ to be the subset $\XT$ of $X^\Z$  containing all bi-infinite sequences $(x_{k})_{k\in\Z}$ such that $T(x_k)=x_{k+1}$ for every $k\in\Z$. Note that the surjectivity of $T$ guarantees that $\XT$ is not empty.
It is easy to see that $\XT$ is a  subsystem of %closed and $S$-invariant subset of 
$X^\Z$.

The \emph{natural extension of $(X,T)$} is the  topological dynamical system obtained by considering $\XT$ with the dynamics of the shift $S$. 
If $T$ is a homeomorphism, then the topological dynamical systems $(X,T)$ are $(\XT,S)$ are topologically conjugate (isomorphic). If $T$ is surjective but not invertible, then the natural extension $(\XT,S)$ of $(X,T)$ is the smallest invertible topological dynamical system that factors onto $(X,T)$.
%, because then we are able to extend our sequences backward, that is, given $x_0$ we can find $x_{-1}$ such that $T(x_{-1})=x_0$, etc. 
%Clearly, $\XT \subset X^\Z$. 

The following is folklore (cf. \cite{CliTho12}).

\begin{lemma} \label{lem:xtmeasures}
The simplices of invariant measures $\MT(X)$ of $(X,T)$ and $\MS(\XT)$ of $(\XT,S)$ are affinely homeomorphic.  
\end{lemma}

Now we would like to define the space of chains. Fix $\delta\geq 0$.
\begin{defn}
    We say that $(x_k)_{k\in\Z}\in X^\Z$ is a \emph{(biinfinite) $\delta$-chain} for $T$ if we have $\rho\left( T(x_k),x_{k+1} \right)\leq\delta$ for every $k\in \Z$. We write $\XTd$ to denote the set of all $\delta$-chains for $T$. Fix $x,y\in X$. A \emph{(finite) $\delta$-chain for $T$ (from $x$ to $y$)} is a sequence $(x_k)_{k=0}^M$ in $X$ such that $x_0=x$, $x_M=y$, and $\rho\left( T(x_k),x_{k+1} \right)\leq\delta$ for $0\le k<M$. 
\end{defn}
Obviously, every element of $\XT$ is a $0$-chain. Note that for every $\delta\ge 0$ the set $\XTd$ is a subsystem of %a topological dynamical system 
$(X^\Z,S)$. % which means that $\XTd$ is an $S$-invariant ($S(\XTd)=\XTd$) nonempty closed subset of $X^\Z$. 
Note also that our assumption on the diameter of $X$ means that $\XT^1=X^\Z$. % because for any sequence $(x_k)_{k\in\Z}\in X^\Z$ we have $\rho\left(T(x_k),x_{k+1}\right)\leq 1$.

%For $n\in \N$, we can consider $1/n$-chains. This will give us 
As a result, we obtain a sequence of approximations of $\XT$ given by a decreasing sequence of subsystems. The $n$th element of this sequence is the subsystem formed by all $1/n$-chains ($n=1,2,\ldots$), that is
\[
    \XT=\XT^0 \subseteq \ldots \subseteq \XT^{1/n}\subseteq \ldots \subseteq \XT^{1/2} \subseteq \XT^1=X^\Z.
\]
 In addition, we have \[\XT=\bigcap\limits_{n=1}^\infty \XT^{1/n}.\]
%\begin{prop}
%For every $\delta>0$ the topological dynamical system has the shadowing property.
%\end{prop}
%\begin{proof}
%Fix $\delta>0$. Pick $\eps>0$. Let $(x^{(n)})_{n=1}^\infty\subseteq\XTd$ be a positive $\Delta$-chain for $S$. Therefore for every $n\ge 1$
%\[
%\rZ((x_k^{(n)})_{k\in\Z},(x_{k+1}^{(n+1)})_{k\in\Z})\le \Delta.
%\]
%\end{proof}
\begin{defn}
    We say that $(X,T)$ is \emph{chain mixing} if for every $\delta>0$ there exists $M\geq 1$ such that for every $x,y\in X$ and for every $n\geq M$ there exists a finite $\delta$-chain $(z_i)_{i=0}^n$ from $z_0=x$ to $z_n=y$.
\end{defn}

%For every metric $\rho$ on a space $X$, the following metric generates the same topology and is bounded by $1$. For $x,y \in X$ we set
%\[
%    \rho_1(x,y)=\min\left\{ 1,\rho(x,y) \right\}.
%\]
%Hence, $\rho_1$ is a bounded metric whatever the original metric was, $\rho_1$ generates the same topology. Therefore, we always assume that $\diam X \leq 1$.
%
\begin{defn}
    For $\delta>0$ we say that a sequence $(x_j )_{j=0}^\infty$ is a \emph{$\delta$-pseudo-orbit} for $T$ if $\rho\left(T(x_j),x_{j+1}\right)<\delta$ for all $j\geq 0$. 
\end{defn}
\begin{defn}
    We say that $(X,T)$ has the \emph{shadowing property} (or the \emph{pseudo-orbit tracing property}) if for every $\eps>0$ there exists $\delta>0$ such that for every $\delta$-pseudo-orbit $(x_j)_{j=0}^\infty\subset X$ there is a point $z\in X$ with $\rho\left( T^j(z), x_j \right)<\eps$ for every $j\geq 0$. In addition, we say that the pseudo-orbit $(x_j )_{j=0}^\infty$ is \emph{$\eps$-traced by $z$}. 
\end{defn}
\begin{defn}
    A sequence $(x_j)_{j=0}^\infty\subset X$ is called a \emph{$\delta$-average pseudo-orbit} if there exists $N\in \N$ such that for each $n\geq N$ and $k\in \N_0$ we have 
    \begin{equation}
        \dfrac{1}{n}\sum_{j=0}^{n-1}\rho\left( T(x_{j+k}),x_{j+k+1}\right)<\delta.
    \end{equation}
\end{defn}
\begin{defn}
    A topological dynamical system $(X,T)$ has the \emph{average shadowing property} if for every $\eps>0$ there exists $\delta>0$ such that for every $\delta$-average pseudo-orbit $(x_j)_{j=0}^\infty\subset X$ there is a point $z\in X$ that satisfies 
    \begin{equation}
        \limsup_{n\rightarrow\infty}\dfrac{1}{n}\sum_{j=0}^{n-1}\rho\left( T^j(z), x_j\right)<\eps.
    \end{equation}
    In this case, we say that the pseudo-orbit $(x_j)_{j=0}^\infty$ is \emph{$\eps$-traced on average by $z$}.
\end{defn}
\begin{thm}[Thm. 3.6 in {\cite{KKO14}}]
    For a chain mixing topological dynamical system $(X,T)$, the following conditions are equivalent:
    \begin{enumerate}
        \item $(X,T)$ has the average shadowing property,
        \item for every $\varepsilon > 0$ there is $\delta > 0$ such that for every $\delta$-pseudo-orbit $(x_i)_{i=0}^{\infty}$ there is a point $z\in X$ such that
        \[ \limsup_{n\rightarrow\infty}\dfrac{1}{n}\sum_{j=0}^{n-1}\rho\left(T^j(z),x_j \right)<\eps.\]
    \end{enumerate}
\end{thm}
\begin{thm}[Thm. 4.5. in {\cite{CT}}] A surjective topological dynamical system $(X,T)$ has the asymptotic average shadowing property if and only if it is Besicovitch complete and has the average shadowing property.
\end{thm}
\begin{defn}
    A sequence $\{x_j\}_{j=0}^\infty\subset X$ is called an \emph{asymptotic average pseudo-orbit} if 
    \begin{equation}
        \lim_{n\rightarrow\infty}\dfrac{1}{n}\sum_{j=0}^{n-1}\rho\left( T(x_j), x_{j+1} \right)=0.
    \end{equation}
\end{defn}
\begin{defn}
    A topological dynamical system $(X,T)$ has the \emph{asymptotic average shadowing property} if for every asymptotic average pseudo-orbit $\{x_j\}_{j=0}^\infty\subset X$ there exists a point $z\in X$ such that 
    \begin{equation}
        \lim_{n\rightarrow\infty}\dfrac{1}{n}\sum_{j=0}^{n-1}\rho\left(T^j(z),x_j \right)=0.
    \end{equation}
\end{defn}

Bowen introduced the specification property in \cite{Bowen71} to examine Axiom A diffeomorphisms. Roughly speaking, this property is a stronger (uniform) form of topological transitivity. In simple terms, a topological dynamical system has the specification property if for every $\varepsilon > 0$, one can find a nonnegative integer $k$ such that for any set of finite orbit segments, there is a point that tracks each of these segments within an $\varepsilon$ distance over the set of times defining that segment and takes $k$ steps to transition between successive orbit segments. To state the definition we will need some vocabulary.

\begin{defn} %Let $(X,T)$ be a topological dynamical system. 
Given integers $0\le a<b$, an interval $[a,b)=\{n\in\N: a\le n <b\}\subseteq\N$ and $x\in X$ we write  $T^{[a,b)}(x)$ to denote the sequence  $(T^i(x))_{a\leq i <b}$ and call it the \emph{orbit segment} (of $x$ over $[a,b)$).  Whenever $T$ is invertible, we allow $[a,b)$ to be a subset of $\Z$, that is, for $a,b\in\Z$ with $a<b$ we set $[a,b)=\{n\in\Z: a\le n <b\}\subseteq\Z$ and we write  $T^{[a,b)}(x)$ to denote the sequence  $(T^i(x))_{a\leq i <b}$. 
\end{defn}

\begin{defn}
Let $k\in\N$. A \emph{$k$-spaced specification} is a sequence of $n\ge 2$ orbit segments $(T^{[a_i,b_i)}(x_i))_{1\leq i\leq n}$ such that $a_i-b_{i-1}\ge k$ for $2\leq i\leq n$. 
\end{defn}

\begin{defn}
Let $\varepsilon>0$. A specification  $(T^{[a_i,b_i)}(x_i))_{1\leq i\leq n}$ is \emph{$\varepsilon$-traced} if there exists $y\in X$ such that $\rho(T^j(x_i),T^j(y))\leq \varepsilon$ for $j\in [a_i,b_i)$ and for every $1\leq i\leq n$. 
\end{defn}

\begin{defn}
    A tds $(X,T)$ has the \emph{periodic specification property} if for every $\varepsilon>0$ there exists $k=k(\varepsilon)\in\mathbb{N}$ such that every $k(\varepsilon)$-spaced specification $(T^{[a_i,b_i)}(x_i))_{1\leq i\leq n}$ with $a_1=0$ is $\varepsilon$-traced by a periodic point $y\in X$ such that $T^{b_n+k(\varepsilon)}(y)=y$.
\end{defn}

An important observation is that chain mixing of the topological dynamical system $(X,T)$ is equivalent to the fact that the periodic specification property holds for $\XTd$ and any $\delta$. 

\begin{prop}\label{defspec}
Let $(X,T)$ be a topological dynamical system. The following are equivalent.
\begin{enumerate}[label=\textnormal{(\roman*)}]
\item \label{i} The tds $(X,T)$ is chain mixing.
    \item \label{ii} For every $\delta>0$ the system $(\XTd,S)$ has the periodic specification property.
    \item \label{iii} For every $\delta>0$ the system $(\XTd,S)$ is topologically mixing.
\end{enumerate}
\end{prop}
\begin{proof}
\noindent \ref{i}$\Leftrightarrow$\ref{iii} This is well known.
\noindent \ref{ii}$\Rightarrow$\ref{iii} This holds because periodic specification implies topological mixing for any tds.
\noindent \ref{i}$\Rightarrow$\ref{ii} Fix $\eps>0$. Let $N>0$ be such that $1/N\le \eps$. Using \eqref{cond:dist-bound}, we see that for $(x_{k}),(y_{k})\in X^\Z$ and $a,b\in \Z$ with $a<b$ we have \begin{equation}
    \label{distance-cond}
\rho(x_{n},y_{n})<\eps\text{ for all $n\in \Z$ such that }a-N< n< b+N-1
\end{equation}
implies that 
$\rZ(S^j(x_{k}),S^j(y_{k}))<\eps$ for every $j\in[a,b)\cap\Z$. 
%Cover $X$ with finitely many nonempty open sets $U_1,\ldots,U_l$ such that each of them has diameter strictly less than $\delta/2$. 
Use chain mixing to find a single number $M\ge 0$ such that for every $x,y\in X$ and every $m\ge M$ %and each pair of indices $1\le i, j \le l$ 
there is a finite $\delta$-chain $(c^{x,y}_s)_{0\le s \le m}$ from $x$ to $y$, that is, $(c^{x,y}_s)_{0\le s \le m}$ is a finite sequence in $X$ such that $c^{x,y}_0=x$, $c^{x,y}_m=y$, and $\rho(T(c^{x,y}_{s-1}),c^{x,y}_{s})\le \delta$ for $1\le s\le m$. %Since $U_1,\ldots,U_l$ cover $X$, for every $x\in X$ we have a well-defined number $u(x)=\min(1\le j \le l: x\in U_j)$.
We claim that with $k(\eps)=2N-2+M$, every $k(\eps)$-spaced specification is $\eps$-traced by a periodic point in $\XTd$. Take $x^{(1)},\ldots,x^{(n)}\in \XTd$ and $k(\eps)$-spaced specification $(S^{[a_i,b_i)}(x^{(i)}))_{1\leq i\leq n}$ such that $a_1=0$. Without loss of generality, we assume $a_{i}-b_{i-1}=k(\eps)$ for $2\le i \le n$. Fix $1\le i \le n$. By definition of $\XTd$, the sequence $x^{(i)}_{(a_i-N,b_i+N-1)}$ is a finite $\delta$-chain. 
Let $x'_i=x^{(i)}_{a_i-N+1}$ be the first and $y'_i=x^{(i)}_{b_i+N-2}$ be the last entry of this chain. %Let $\gamma_i=b_i-a_i+2N-2$. Note that $y'_i$ is the last element of the orbit segment $T^{[0,\gamma_1)}(x'_i)$.
Define $x'_{n+1}=x'_{1}$. For $1\le i \le n$, let $c^{(i)}_{[0, M]}=(c^{y'_i,x'_{i+1}}_s)_{0\le s \le M}$ be $\delta$-chain from $y'_i$ to $x'_{i+1}$.
%, whose first element $c^{(i)}_0$ belongs to the same set $U\in\{U_1,\ldots,U_l\}$ as $T(y'_i)$ and %the $T$-image of 
%the last element, $c^{(i)}_M$ %$T(c^{(i)}_M)$ 
%is in the same set $V\in\{U_1,\ldots,U_l\}$ as $x'_{i+1}$. Since the diameter of $U$ and $V$ is less than $\delta/2$, replacing the last entry $c^{(i)}_M$ of the chain $c^{(i)}_{[0, M]}$ for $1\le i \le n$ by $x'_{i+1}$ for $1\le i \le n$, we still get $\delta$-chain. Similarly, if we form the sequence 
%\[x^{(i)}_{(a_1-N,b_1+N-1)},c^{(i)}_{(0,M)]},\]
%where $1\le i\le n$, the we also get $\delta$-chain
Therefore the concatenation of sequences
\begin{equation}\label{per-d-chain}
    x^{(1)}_{(a_1-N,b_1+N-1)}, c^{(1)}_{(0, M)},\ldots, x^{(n)}_{(a_n-N,b_n+N-1)}, c^{(n)}_{(0, M)} 
\end{equation}
is $\delta$-chain. Furthermore, repeating periodically the sequence \eqref{per-d-chain} we obtain a doubly infinite periodic $\delta$-chain denoted $(y_k)$ (we choose the position $y_0$ so that it corresponds to $N-1$ entry in the sequence \eqref{per-d-chain}, that is $y_0=x^{(1)}_{a_1}$). It follows from the construction that $S^{b_n+k(\eps)}(y_k)=(y_k)$. 
Hence $(y_k)$ is $S$-periodic. Furthermore it belongs to $\XTd$ because  applying $T$ to the last entry of our sequence \eqref{per-d-chain}, which is $c^{(n)}_{M-1}$ we get a point that is at most $\delta/2$ close to the point $c^{(n)}_{M}$, which is at most $\delta/2$ close to $x'_{n+1}=x'_1$, which is the first entry of \eqref{per-d-chain}. Finally, we note that for every $1\le i\le n$ we have
\[
\rho(x^{(i)}_{n},y_{n})=0<\eps\text{ for all $n\in \Z$ such that }a_i-N< n< b_i+N-1
\]
because we have $x^{(i)}_{n}=y_{n}$ for these $n$. Therefore \eqref{distance-cond} holds for $1\le i \le n$, so
\[
\rZ(S^j((x^{(i)}_k)),S^j((y_k)))<\eps
\]
for every $j\in [a_i,b_i)\cap\Z$. 
\end{proof}

\section{Examples}
%Explain and provide examples of systems with the AASP but without any stronger specification.

The aim of this section is to point out that our main theorem proves density of ergodic measures for many topological dynamical systems that are not covered by known results. %In other words, we point out, that

By \cite{CT} we know that the asymptotic average shadowing and vague specification properties are equivalent. Therefore, assuming either the weak specification property or the almost specification property, we conclude that the vague specification property holds (see \cite{KKO14,KLO17}). Hence, any factor of a topological dynamical system with the weak specification property or with the almost specification property has the vague specification property, since the vague specification is inherited by factors (see \cite{Kamae}).

In addition, we also observe that finite products of spaces with the vague specification property have the vague specification property. 

\begin{prop}\label{prop:vsp-prod}
If $(X,T)$ and $(Y,S)$ are topological dynamical systems with the vague specification property, then the product tds $(X\times Y,T\times S)$ also has the vague specification property.   
\end{prop}
\begin{proof}
    We will show that $(X\times Y, T\times S)$ has the asymptotic average shadowing property since the vague specification property is equivalent to the asymptotic average shadowing property. Let $\{z_j\}_{j=1}^\infty$ be an asymptotic average pseudo-orbit in $X\times Y$. For $j\geq 1$, we write $z_j=(x_j,y_j)$ where $x_j\in X$, $y_j\in Y$. Then we see that $\{x_j\}_{j=1}^\infty\subset X$ and $\{ y_j\}_{j=1}^\infty \subset Y$ are asymptotic average pseudo-orbits in $X$, and $Y$, respectively. Hence, there exist $x\in X$ and $y\in Y$ such that 
    \begin{align}\label{eqn:tracing-pseudo-orbits}
        \begin{split}
            & \lim\limits_{n\rightarrow \infty }\dfrac{1}{n}\sum_{j=0}^{n-1}\rho_X\left(T^j(x),x_j\right)=0,\\
        &\lim\limits_{n\rightarrow \infty }\dfrac{1}{n}\sum_{j=0}^{n-1}\rho_Y\left(T^j(y),y_j\right)=0.
        \end{split}
    \end{align} 
    Using \eqref{eqn:tracing-pseudo-orbits} we get $z=(x,y)\in X\times Y$ such that 
    \begin{equation*}
        \lim\limits_{n\rightarrow \infty }\dfrac{1}{n}\sum_{j=0}^{n-1}\rho_{X\times Y}\left((T\times S)^j(z),z_j\right)=0,
    \end{equation*}
    since $\rho_{X\times Y}$ is uniformly equivalent to the $\rho_{\max}(z_1,z_2):=\max\{\rho_X(x_1,x_2),\rho_Y(y_1,y_2)\}$, and $\max\{a,b\}\leq a+b$ for $a,b\geq 0$.
    Thus, $(X\times Y, T\times S)$ has the asymptotic average shadowing property.
\end{proof}

%Using the properites of the vague specification property we get the following conclusion. 

\begin{cor}\label{cor:prod-vsp}
    Assume that $(X,T)$ has the weak specification property but not the almost specification property, and $(Y,S)$ has the almost specification property but not the weak specification property. Then the product tds $(X\times Y, T\times S)$ has the vague specification property but neither the weak specification property nor the almost specification property.
\end{cor}
\begin{proof}
    As discussed above, both systems $(X,T)$ and $(Y,S)$ have the vague specification property. The same holds for $(X\times Y, T\times S)$ by Proposition \ref{prop:vsp-prod}. Now, if $(X\times Y, T\times S)$ has the weak specification property then so does $(Y,S)$ since $(Y,S)$ is a factor of $(X\times Y, T\times S)$ and the weak specification property is inherited by factors (see \cite{KLO16}). Similarly, $(X\times Y, T\times S)$ cannot have the almost specification property. We conclude that the product tds $(X\times Y, T\times S)$ has the vague specification property but has neither the weak specification property nor the almost specification property. 
\end{proof}
By Corollary \ref{cor:prod-vsp} we get examples of non-symbolic topological dynamical systems having the vague specification property but without the weak specification property or the almost specification property. Since the example of a topological dynamical system $(X,T)$ with the weak specification property but without the almost specification property is not a subshift (it is not positively expansive, so it cannot be a subsystem of any symbolic system) neither is any product system having $(X,T)$ as one of its factors.

%It is enough to  take a system with 

Note that any non-trivial system with either the weak specification property or the almost specification property has many disjoint invariant subsets. Therefore, a non-trivial system with weak specification (or almost specification) is neither proximal nor minimal. Proximality is impossible because proximal systems do not have disjoint subsystems. Minimality is excluded because minimal tds do not have any non-trivial subsystems. Hence, the examples of proximal and minimal shift spaces that were constructed in \cite{CKKK} and examined further in \cite{CT} cannot have the weak specification property or the almost specification property. However, they both have the vague specification property (see \cite{CT}).

%\section{Approximations}
%Definition

%\[
%\XT=\bigcap_{n=1}^\infty \XTn.
%\]

%Our main observations is that chain mixing of the system $(X,T)$ implies that $\XTd$ has the periodic specification property of $\delta>0$. 
%\begin{prop}
%If $(X,T)$ is chain mixing, then for every $\delta>0$ the system $\XTd$ has the periodic specification property.    
%\end{prop}
%\begin{proof}
%\end{proof}

\section{Different shades of the Besicovitch pseudometric}
When it comes to defining the Besicovitch pseudometric on the product space $X^\Z$ we face a choice. Taking the topological dynamical system $(X^\Z,S)$ and a metric $\rZ$ on $X^\Z$ we follow the general procedure of defining the Besicovitch pseudometric on a topological dynamical system. We call the resulting pseudometric the dynamical Besicovitch pseudometric on $X^\Z$.
\begin{defn}
    The \emph{dynamical Besicovitch pseudometric} on $\XZ$ is defined to be the $\rB$-distance between orbits in the system $(\XZ,S)$, that is,
        \begin{equation} \label{eqn:bes2}
            \rZB(x,y) = \limsup_{N\to\infty} \frac{1}{N}\sum_{j=0}^{N-1}\rZ(S^j(x),S^j(y)).
        \end{equation}
\end{defn}
The other option ignores the dynamics induced by the shift $S$ on $X^\Z$ and is based on the fact that the elements of $X^\Z$ are sequences with values in a metric space $(X,\rho)$. This leads to a pseudometric we call the coordinatewise Besicovitch pseudometric on $X^\Z$. In fact, the dynamical Besicovitch pseudometric is based on the same principle, but it uses that the orbits of points in $X^\Z$ with respect to the shift $S$ are sequences of elements of $X^\Z$, that is, elements of $(X^\Z)^\Z$.
\begin{defn}
    The \emph{coordinatewise Besicovitch pseudometric} $\crB$ on $\XZ$ is defined for $x=(x_{j})_{j\in\Z},y=(y_{j})_{j\in\Z} \in \XZ$ by
        \begin{equation} \label{eqn:bes1}
            \rB(x,y) = \limsup_{N\to\infty} \frac{1}{N}\sum_{j=0}^{N-1}\rho(x_j,y_j).
        \end{equation}
\end{defn}
Since every point in $X^\Z$ contains complete information about the orbit of that point with respect to the shift $S$ it is not surprising that the coordinatewise and the dynamical Besicovitch pseudometrics are uniformly equivalent on $X^\Z$. 
\begin{lemma}\label{lem:coordinate-dynamical-equiv}
    The coordinatewise Besicovitch pseudometric $\rB$ and the dynamical Besicovitch pseudometric $\rZB$ are uniformly equivalent on $\XZ$.
\end{lemma}

\begin{proof}
    By \cite[Lemma 2]{KLO17}, $\rB$ is uniformly equivalent to the pseudometric $\hatrB$ on $X^\Z$ defined by
        \[
            \hatrB(x,y) = \inf\{\delta>0 \mid \dbar(\{k\ge0 \mid \rho(x_k,y_k)\ge\delta\}) < \delta\},
        \]
    where as usual, $\dbar(A)$ is the upper asymptotic density of $A\subset\N_0$, that is
        \[
            \dbar(A)= \limsup_{n\to\infty} \frac{\#(A\cap\{0,\dots,n-1\})}{n}.
        \]
 Applying the same lemma to the metric space $(\XZ,\rZ)$, we also have an equivalence between $\rZB$ and the pseudometric $\hatrBZ$ defined by
        \[
            \hatrBZ(x,y) = \inf\{\delta>0 \mid \dbar(\{k\ge0 \mid \rZ(S^k(x),S^k(y))\ge\delta\}) < \delta\}.
        \]
    To complete the proof, we show that $\hatrB$ is equivalent to $\hatrBZ$. One direction is immediate, since $\rZ(S^k(x),S^k(y)) \ge \rho(x_k,y_k)$ for any $k\ge0$. (Indeed, if $\delta \in(0,1]$ and $N$ is the largest integer dominated by $1/\delta-1$, then $\rho(x_k,y_k) \ge \delta$ implies that $\rZ(S^n(x),S^n(y)) \ge \delta$ for $k-N \le n \le k+N$ and in particular that $\rZ(S^k(x),S^k(y)) \ge \delta$.) So for any $\delta >0$, $\rZ(S^k(x),S^k(y)) \ge \delta$ is a more frequent event than $\rho(x_k,y_k) \ge \delta$, and so $\hatrBZ(x,y)$ dominates $\hatrB(x,y)$ as it is the infimum of a smaller set.

    For the other direction, let $\delta\in(0,1]$, and as above let $N\in\N$ be the largest integer dominated by $1/\delta-1$. Set $\delta'=\frac{\delta}{2N+1}$. Suppose that $x,y\in\XZ$ with $\hatrB(x,y)<\delta'$, so
        \begin{equation} \label{eqn:dbar}
            \dbar(\{k\ge0 \mid \rho(x_k,y_k)\ge\delta'\}) < \delta/(2N+1).
        \end{equation}
    If $k\ge0$ is such that $\rZ(S^k(x),S^k(y)) \ge \delta'$, then there exists $n\in\Z$ with $k-N \le n \le k+N$ such that $\rho(x_n,y_n)\ge\delta'$. In other words, $\{k\ge0 \mid \rZ(S^k(x),S^k(y))\ge \delta'\}$ is a subset of
        \[
            \{k\ge0 \mid \rho(x_n,y_n)\ge\delta' \text{ for some } n\in\{k-N,\dots,k+N\}\}.
        \]
    It follows that the upper asymptotic density of $\{k\ge0 \mid \rho(x_k,y_k) \ge \delta'\}$ is at least
        \[
            (1/(2N+1)) \cdot \dbar(\{k\ge0 \mid \rZ(S^k(x),S^k(y))\ge \delta'\}). 
        \]
    Since $\delta\ge\delta'$, we therefore have from \eqref{eqn:dbar} that
        \begin{align*}
	   \dbar(\{k\ge0 \mid \rZ(S^k(x),S^k(y))\ge \delta\}) &\le \dbar(\{k\ge0 \mid \rZ(S^k(x),S^k(y))\ge \delta'\})\\
	   &\le (2N+1) \dbar(\{k\ge0 \mid \rho(x_k,y_k) \ge \delta'\})\\
	   &< \delta
	\end{align*}
    and so by definition, $\hatrBZ(x,y) \le \delta$.
%    let $\varepsilon >0$, choose $N\in\N$ with $2^{1-N}<\varepsilon$, and suppose that $x,y\in\XZ$ with $\hatrB(x,y)<\varepsilon'=\frac{\varepsilon}{2N+1}$. In particular,
%        \begin{equation} \label{eqn:dbar}
%            \dbar(\{k\ge0 \mid \rho(x_k,y_k)\ge\varepsilon\}) \le \dbar(\{k\ge0 \mid \rho(x_k,y_k)\ge\varepsilon'\}) < \varepsilon'.
%        \end{equation}
%    Observe that, if $k\ge0$ is such that $\rho(x_j,y_j)<\varepsilon$ for every $j\in\{k-N,\dots,k+N\}$, then
%        \begin{align*}
%            \rZ(S^k(x),S^k(y))
%                &= \sum_{j\in\Z} 2^{-|j|}\rho(x_{j+k},y_{j+k})\\
%                &\le \sum_{j=-N}^N2^{-|j|}\rho(x_{k+j},y_{k+j}) + 2 \cdot \sum_{j=N+1}^\infty2^{-j}\diam(X)\\
%                &< 4\varepsilon.
%        \end{align*}
%    In other words, $\{k\ge0 \mid \rZ(S^k(x),S^k(y))\ge 4\varepsilon\}$ is a subset of
%        \[
%            \{k\ge0 \mid \rho(x_j,y_j)\ge\varepsilon \text{ for some } j\in\{k-N,\dots,k+N\}\}.
%        \]
%    The upper asymptotic density of this latter set is at most
%        \[
%            (2N+1) \cdot \dbar(\{k\ge0 \mid \rho(x_k,y_k) \ge \varepsilon\}),
%        \]
%    which by \eqref{eqn:dbar} is less than $(2N+1)\varepsilon' = \varepsilon < 4\varepsilon$. It follows that
%        \[
%            4\varepsilon \in \{\delta>0 \mid \dbar(\{k\ge0 \mid \rZ(S^k(x),S^k(y))\ge \delta\}) < \delta\},
%        \]
%    so by definition, $\hatrBZ(x,y) \le 4\varepsilon$.
\end{proof}

\section{$\bar\rho$}

Usually we endow the space of invariant measures $\MT(X)$ of a tds $(X,T)$ with the weak$^*$ topology. This topology has some shortcomings.  It is not compatible with entropy, that is the entropy function is not continuous on $\MT(X)$  (high-entropy measures can be weak$^*$ close to low-entropy measures, in particular zero-entropy measures are often dense, but the entropy of some measures is non-zero). Furthermore, properties like mixing or ergodicity are not preserved by weak$^*$ convergence. The great advantage of the weak$^*$ topology on $\MT(X)$ is that it is compact and metrisable. 

When $(X,T)$ is a subsystem of a symbolic system, an alternative topology on $\MT(X)$ is provided by Ornstein's d-bar metric $\dbar$. Ornstein's metric on $\MT(X)$ is complete and better captures dynamical properties, although it is typically harder to work with. Furthermore, $\MT(X)$ endowed with Ornstein's metric often is neither compact nor even separable. Up to recently, Ornstein's metric has been so far studied only for symbolic systems (\cite{MR368127} and \cite{Schwarz} are the only exceptions we know). 
For symbolic systems, Ornstein's d-bar distance between two shift-invariant measures $\mu$ and $\nu$ roughly measures how well you can match the typical (generic) orbits for $\mu$ and $\nu$. To measure how well two points (symbolic sequences) are matched, one uses the upper asymptotic density of the set of position at which the sequences differ. In a forthcoming paper \cite{BCKO} 
%"Spectrum of Invariant Measures via Generic Points," 
the authors study a new metric $\bar{\rho}$ on the space $\MT(X)$ of invariant measures of a topological dynamical system $(X,T)$.
The metric is defined by $\bar{\rho}(\mu,\nu) = \inf\{\int\rho(x,y)d\lambda(x,y) : \lambda \in J(\mu,\nu)\}$, where $J(\mu,\nu)$ is the set of all joinings of $\mu$ and $\nu$, and $\rho$ is a compatible metric on $X$. 
The metric $\bar{\rho}$ is uniformly equivalent to Ornstein's d-bar metric on the space of shift-invariant measures on symbolic spaces. However, $\bar{\rho}$ works for general topological dynamical systems rather than just symbolic ones.

Here we only recall the definition of $\bar{\rho}$, for details we refer to \cite{BCKO}. We need to introduce some terminology. 
Assume that $(X,T)$ is a topological dynamical system. Let $\projX$ and $\projY$ denote the canonical projections from $X\times X$ onto the first and the second coordinate. For any measure $\lambda$ on $X\times X$, we define the push-forward measure $\projZ_*(\lambda)$ on a Borel set $B\subseteq X$ (where $i$ is either $1$ or $2$) by $\projZ_*(\lambda)(B)=\lambda((\projZ)^{-1}(B))$.
A \emph{joining} of $\mu,\nu\in \MT(X)$ is defined as a $T\times T$-invariant measure $\lambda$ on $X\times X$ satisfying $\projX_*(\lambda)=\mu$ and $\projY_*(\lambda)=\nu$. The set of all such joinings, denoted by $J(\mu,\nu)$, forms a nonempty, closed, and convex subset of $\mathcal{M}_{T\times S}(X\times X)$. We set
\begin{equation}\label{def:rho-bar}
\bar{\rho}(\mu,\nu)=\inf_{\lambda\in J(\mu,\nu)}\int_{X\times X} \rho(x,y) \text{ d}\lambda(x,y).
\end{equation}

The map $J(\mu,\nu)\ni\lambda\mapsto \int_{X\times X} \rho(x,y) \text{ d}\lambda(x,y)\in\R$ is both convex and continuous. Consequently, the infimum in \eqref{def:rho-bar} is attained at some extreme point of $J(\mu,\nu)$. Moreover, when both $\mu$ and $\nu$ are ergodic, the set $J^e(\mu,\nu)$ of extreme points of $J(\mu,\nu)$ consists entirely of ergodic $T\times T$-invariant measures.
If $\rho$ is a compatible metric on $X$, then $\bar{\rho}$ defined by \eqref{def:rho-bar} is a metric on $\MT(X)$ (cf. the proof of \cite[Theorem 21.1.2]{Garling}). The triangle inequality relies on the gluing lemma \cite[Theorem 16.11.1]{Garling}. The topology induced by $\bar{\rho}$ on $\MT(X)$ is stronger than the weak$^*$ topology: If $\mu\in\MT(X)$ and $(\mu_n)_{n\ge 1}$ is a sequence in $\MT(X)$ such that $\bar{\rho}(\mu_n,\mu)\to 0$ as $n\to\infty$, then $\mu_n$ converges to $\mu$ in the weak$^*$ topology.
For a proof, write $W_1$ for the Wasserstein metric that induces the weak$^*$ topology (see \cite{Garling}), and note that $W_1(\mu,\nu)\le \bar{\rho}(\mu,\nu)$ for all $\mu,\nu\in\MT(X)$. %We formalize this observation in the following remark:

%\begin{remark}\label{rem:weakstar}
%\end{remark}

There is a strong connection between Besicovitch dynamically defined pseudometric and $\bar\rho$. This connection allows the authors of \cite{BCKO} to show that if $(X,T)$ is a tds, $\mu,\nu\in\MT(X)$, and $x,y\in X$ are such that $x$ (respectively, $y$) is generic for $\mu$ (respectively, $\nu$), then
\begin{equation}\label{ineq:B-bound-on-bar-rho}
\bar{\rho}(\mu,\nu)\le
\limsup_{n\to\infty}
\frac{1}{n} \sum_{j=0}^{n-1} \rho(T^j(x),S
T^j(y)).    
\end{equation}
This allows the authors of \cite{BCKO} to translate properties of sequences of generic points for invariant measures converging in the Besicovitch pseudometric to properties of sequences of measures converging in the $\bar{\rho}$ metric. It follows that many properties of invariant measures are $\bar\rho$-closed.

\section{Hausdorffication of $\bar\rho$}

We extend the results obtained in \cite{KKK} for symbolic systems to general dynamical systems.
Let $(X,T)$ be a topological dynamical system.
Having a well-behaved distance $\bar{\rho}$ between $T$-invariant measures on $X$ accompanied by the Besicovitch pseudometric $\rB$ between points in $X$, we turn to defining the distances between sets of measures and sets of points. 
To this end, we mimic the construction of the Hausdorff metric between compact subsets of a metric space. Recall that the Hausdorff metric measures the (pseudo-) distance between two sets by finding the supremum of (pseudo-) distances from points in one set to the other set. Note, however, that in our setting we no longer consider compact sets; usually the spaces of all $T$-invariant measures endowed with $\bar{\rho}$ and $X^\Z$ with the Besicovitch pseudometric are not compact. However, the set of possible distances is bounded above, so we still get a Hausdorff metric $\barrH$ for sets of measures or a Hausdorff pseudometric $\rBH$ for subsets of $X$. See \cite[pp. 11--15]{hyperspace-book} for more details. In particular, both constructions apply to the topological dynamical system $(X^\Z,S)$. 

Formally, for nonempty sets $A$ and $B$ contained in a set $X$ endowed with a bounded pseudometric $p$ we set
\[
p^H(A,B) = \max\{ \sup_{a\in A} \inf_{b\in B} p(a,b), \sup_{b\in B} \inf_{a\in A} p(a,b) \}.
\]
This leads to a pseudometric that captures how far apart two sets are in terms of their worst-case point-to-set distances in both directions.
Recall that the distance $p(x,A)$ between a point $x$ and a nonempty subset $A$ in a pseudometric space $(X,p)$ is defined as:
\[
p(x,A) = \inf\{p(x,a) : a \in A\}.
\]
This measures the largest lower bound for a pseudo-distance from $x$ to a point in $A$. If $p$ is a metric on $x$, then $p(x,A)=0$ if and only if $x\in \overline{A}$, so in this case, $p^H$ is a metric on the space of all closed nonempty subsets of $X$.

Therefore, given a topological dynamical system $(X^\Z,S)$ we endow $X^\Z$ with a metric $\rZ$. Hence, we obtain a metric $\brZ$ on $\MS(X^\Z)$ and a pseudometric $\rZB$ on $X^Z$. Using the construction described above, we get the metric $\brZH$ on all nonempty $\brZ$-closed subsets of $\MS(X^\Z)$ and the pseudometric $\rZBH$ on all subsets of $X$. In particular, for all subsystems $Y$ and $Z$ of $X^\Z$ we may compare the Hausdorff pseudo-distance $\rZBH(Y,Z)$ with the distance $\brZH(\MSe(Y),\MSe(Z))$ between their sets of ergodic invariant measures (the latter are nonempty $\brZ$-closed sets in $\MS(X^\Z)$, closedness follows from \cite{BCKO}). %We get

\begin{prop}\label{prop:dbar-le-dund}
If $Y$ and $Z$ are subsystems of $X^\Z$, then
\[\brZH(\MS(Y),\MS(Z))=\brZH(\MSe(Y),\MSe(Z))\leq \rZBH(Y,Z). %\Hdund(X,Y)\le \Hdbar(X,Y).
\]
\end{prop}
The proof follows the same lines as the proof of \cite[Prop. 15]{KKK}.
\begin{proof}
    Let $\varepsilon>0$ with $\rZBH(Y,Z)<\varepsilon$ and let $\mu\in\MSe(Y)$. We show that there is $\nu\in\MS(Z)$ with $\bar\pi(\mu,\nu)<\varepsilon$, and Lemma~\ref{lem:pseudogeneric-dbar} and Lemma~\ref{lem:erg-eq} will then imply that 
    \[\brZH(\MSe(Y),\MSe(Z))=\brZH(\MS(Y),\MS(Z))<\varepsilon.\] Let $y\in Y$ be a generic point for $\mu$, let $z\in Z$ with $\rZB(y,z)<\varepsilon$ and proceed as in the proof of \cite[Theorem 7.2]{BCKO} using \eqref{ineq:B-bound-on-bar-rho} on the way.
\end{proof}

Similarly as in \cite[Lemma 14]{KKK} we discover that for every pair of subsystems $Y$ and $Z$ of $X^\Z$ 
the Hausdorff distance $\brZH(\MS(Y),\MS(Z))$ equals the distance $\brZH(\MSe(Y),\MSe(Z))$ between the sets of ergodic invariant measures.
%Our next observation is the equality \[\barrH(\MS(Y),\MS(Z))=\barrH(\MSe(Y),\MSe(Z)),\]
%that holds for any subsystems $Y$ and $Z$ of $X^\Z$. 
\begin{lemma}\label{lem:erg-eq}
If $Y$ and $Z$ are subsytems of $X^\Z$, then
\[\brZH(\MS(Y),\MS(Z)) = \brZH(\MSe(Y),\MSe(Z)).\]
%Furthermore, both quantities are equal
%\[
%\max\left(
% \sup_{\mu\in\Mse(X)}\dbarm(\mu,\Ms(Y)),\ \sup_{\nu\in\Mse(Y)}\dbarm(\nu,\Ms(X)) \right).
%\]
\end{lemma}

We establish equality in the same way as \cite[Lemma 14]{KKK}, that is, we need the following lemma, which is an analogue of \cite[Lemma 13]{KKK} and has the same proof as the one presented in \cite{KKK}.

\begin{lemma}\label{lem:pseudogeneric-dbar}
If $Y\subseteq X^\Z$ is a subsystem and $\mu \in \MSeXZ$, then
$\brZ(\mu,\MSe(Y))=\brZ(\mu,\MS(Y))$.
\end{lemma}

Our main observation is that $\brZH$-limit of simplices of invariant measures preserves density of ergodic measures.

\begin{thm}\label{thm:gen-scheme}
Let $(Y_k)_{k=1}^\infty$ and $Y$ be subsystems of $(X^\Z,S)$ such that
\[\brZH(\MS(Y_k),\MS(Y)) \to 0 \quad\text{ as } k \to \infty.\]
If $\MSe(Y_k)$ is weak$^*$ dense in $\MS(Y_k)$ for every $k \in \N$, then the ergodic measures are weak$^*$ dense in $\MS(Y)$.
\end{thm}

The proof is (up to necessary, mostly notational) changes identical with the one in \cite{KKK}.
%\end{conj}

\section{Main results}

\begin{prop}\label{prop:approximation-for-vsp}
    Suppose that $(X,T)$ is a surjective topological dynamical system and has the vague specification property. Then
    \[
        \brZH(\MS(\XT^{1/n}),\MS(\XT)) \to 0 \quad\text{ as } n \to \infty.
    \]
\end{prop}
\begin{proof}
    Fix $\eps>0$. By Proposition~\ref{prop:dbar-le-dund}, it is enough to find $\delta>0$ such that $\rZBH(\XT,\XT^\delta)\le\eps$. But first we use Lemma~\ref{lem:coordinate-dynamical-equiv} to find $\eps'>0$ such that for every $x=(x_k)_{k\in\Z}$ and in $X^\Z$ if the coordinatewise Besicovitch pseudodistance satisfies $\rB(x,y)<\eps'$ then $\rZB(x,y)<\eps$.
    
    Surjectivity and vague specification (which by \cite{CT} is equivalent to the asymptotic average shadowing property) guarantee that $(X,T)$ has the average shadowing property \cite[Theorem 3.7]{KKO14} and is chain mixing \cite[Lemma 3.1]{KKO14}. By \cite[Theorem  3.6]{KKO14}, it then follows that there is $\delta>0$ such that every $\delta$-pseudo-orbit (in particular, every element $(y_k)_{k\in\Z}$ of $\XT^\delta$ restricted to non-negative coordinates) is $\eps'$-shadowed on average by some $x\in X$, which means that taking any element $(x_k)_{k\in\Z}$ of $\XT$ such that $x_0=x$ (and hence $x_k=T^k(x)$ for $k\ge 0$) we
    have
    \[
    \rB((x_k)_{k\in\Z},(y_k)_{k\in \Z}) = \limsup_{N\to\infty} \frac{1}{N}\sum_{k=0}^{N-1}\rho(T^k(x),y_k)<\eps'.
    \]
    By our choice of $\eps'$ we conclude that $\rZB((x_k)_{k\in\Z},(y_k)_{k\in \Z})<\eps$. It follows that we have found $\delta>0$ such that for every $(y_k)_{k\in\Z}\in\XT^\delta$ we have $\rZB(\XT,(y_k)_{k\in\Z})<\eps$. Since $\XT\subseteq\XTd$ we conclude that $\rZBH(\XT,\XT^\delta)\le \eps$.
    \end{proof}

Finally, we prove our main theorem.

\begin{thm}\label{thm:main}
    If $(X,T)$ is a surjective topological dynamical system with the vague specification property, then the ergodic measures are dense among all invariant measures. 
\end{thm}

\begin{proof}
    If $(X,T)$ is a surjective topological dynamical system with the vague specification property then $(X,T)$ is chain mixing by applying first \cite[Theorem 4.5]{CT} then \cite[Lemma 3.1]{KKO14}. Using Proposition \ref{defspec} we see that for $n\geq 1$, the space $(X_T^{1/n},S)$ has the periodic specification property that implies the set of ergodic measures are dense in $\MS(\XT^{1/n})$, see \cite[Corollary 25]{KKK}. Now, we use the Proposition \ref{prop:approximation-for-vsp} to get that the set of ergodic measures dense in $\MS(\XT)$. By Lemma~\ref{lem:xtmeasures}, we are done.
\end{proof}

\bibliographystyle{ieeetr}
\bibliography{references}
\end{document}